\documentclass[12pt]{article}
\usepackage{amssymb,amsmath,amsthm}

\usepackage{fullpage}
\usepackage{graphicx}
\usepackage{color}

\newtheorem{theorem}{Theorem}

\newtheorem{corollary}[theorem]{Corollary}

\title{The dib-chromatic number of digraphs}
\author{Nahid Y. Javier-Nol\footnotemark[2] \and  Christian Rubio-Montiel\footnotemark[3] \and Ingrid Torres-Ramos\footnotemark[3]}

\begin{document}
\maketitle

\def\thefootnote{\fnsymbol{footnote}}
\footnotetext[2]{Departamento de Matem{\' a}ticas, Universidad Aut{\' o}noma Metropolitana, Iztapalapa, CDMX. {\tt nahid@xanum.uam.mx}.}
\footnotetext[3]{Divisi{\' o}n de Matem{\' a}ticas e Ingenier{\' i}a, FES Acatl{\' a}n, Universidad Nacional Aut{\'o}noma de M{\' e}xico, Naucalpan, Mexico. {\tt christian.rubio@acatlan.unam.mx}.}

\begin{abstract}
We study an extension to directed graphs of the parameter called the $b$-chromatic number of a graph in terms of acyclic vertex colorings: the dib-chromatic number. We give general bounds for this parameter. We also show some results about tournaments and regular digraphs.
\end{abstract}



\section{Introduction}

In this paper, we introduce the concept of the dib-chromatic number of a digraph, which is a generalization of the b-chromatic number of a graph. The term $b$-chromatic number of a digraph is reserved for the $b$-chromatic number of the underlying graph of a given digraph; see \cite{MR3511873}.

To begin with, we recall concepts about graphs. A coloring of a graph $G$ is \emph{proper} if each color class induces an empty subgraph. The \emph{chromatic number} $\chi(G)$ of $G$ is the smallest natural number $k$ such that $G$ admits a proper coloring with $k$ colors.

A \emph{$b$-coloring} of a graph $G$ with $k$ colors is a coloring of the vertices of $G$ such that in each color class there exists a vertex having neighbors in all the other $k-1$ color classes. Such a vertex will be called a \emph{$b$-vertex}. The \emph{$b$-chromatic number} $b(G)$ of a graph $G$ is the largest $k$ for which $G$ has a proper $b$-coloring with $k$ colors \cite{MR1670155}.

This parameter belongs to a family of parameters that comes from complete colorings, namely, the chromatic number, the Grundy number, and the achromatic number (and their no-proper versions). A coloring of a graph $G$ is \emph{complete} if for every pair of colors $i$ and $j$ there exists an edge $uv$ such that $u$ is colored $i$ and $v$ is colored $j$. In fact, any proper coloring of $G$ with $\chi(G)$ colors is complete.

Another way to understand this parameter comes from algorithmic graph theory. Suppose a graph with a given proper coloring. We would like to perform an operation to reduce the number of colors as follows. Recolor its vertices, but not necessarily each one, with the same color. Clearly, such recoloring is not possible if each color class contains a $b$-vertex. Hence, the $b$-chromatic number of the graph serves as the worst case for the number of colors used by this coloring heuristic method.

Next, we define digraph, the dib-chromatic number of a digraph, and we recall some valuable results about the digrundy and diachromatic numbers.

In this paper, we consider simple, loopless, and finite digraphs. A \emph{digraph} $D$ is a pair $(V,A)$ where $V$ is the set of vertices and $A$ is the set of darts. A \emph{dart} (or \emph{arc}) is a pair $(u,v)$ of vertices; for short, we write $uv$, such that $u\not=v$.

A \emph{coloring} of a digraph $D$ is \emph{acyclic} if each color class induces a subdigraph without directed cycles. The \emph{dichromatic number} $dc(D)$ of a digraph $D$ is the smallest number $k$ such that $D$ admits an acyclic coloring with $k$ colors \cite{MR693366}.

A coloring of a digraph $D$ is \emph{complete} if for every pair of different colors $(i,j)$ there exists at least one dart $uv$ such that $u$ is colored $i$ and $v$ is colored $j$ \cite{MR3329642} (see also \cite{MR2998438,MR2895432}). Any acyclic coloring of $D$ with $dc(D)$ colors is also a complete coloring. The diachromatic number $dac(D)$ of $D$ is the largest number of colors for which there exists a complete and acyclic coloring of $D$ \cite{MR3875016}.

Let $u$ and $v$ be two vertices of a digraph $D$ such that $f=uv$ is a dart of $D$. We say that $f$ is \emph{incident from} $u$, and \emph{incident to} $v$, while $u$ is \emph{incident to} $f$ and $v$ is \emph{incident from} $f$. The out-degree $deg ^+(u)$ of a vertex $u$ is the number of darts that are incident from $u$. Similarly, the in-degree $deg^-(v)$ of a vertex $v$ is the number of darts incident to $v$. The degree $deg(v)$ of a vertex $v$ is defined by $deg(v)=deg^+(v)+ deg^-(v)$. We also say that $u$ is \emph{incident to} $v$ and $v$ is \emph{incident from} $u$.

A vertex $u$ is called a \emph{$b^+$-vertex} with respect to a coloring $\varsigma$ of a digraph $D$ if $u$ is incident to a vertex colored with each color different from the color of $u$. Analogously, a vertex $v$ is a \emph{$b^-$-vertex} with respect to $\varsigma$ if $v$ is incident from a vertex colored with each color different from the color of $v$.

A \emph{$b$-coloring} of $D$ is a coloring of $D$ such that each color class contains a $b^+$-vertex and a $b^-$-vertex.
We define the \emph{$dib$-chromatic number} of $D$, denoted by $dib(D)$, as the largest $k$ such that $D$ admits an acyclic $b$-coloring with $k$ colors. The existence of $b$-colorations is given in the parallel paper by one of the co-authors in \cite{montellano2025dib}.

For any digraph $D$, any $b$-coloring of $D$ is also a complete coloring, and then
\begin{equation}\label{eq1}
    dc(D)\leq dib(D)\leq dac(D).
\end{equation}

If $\Delta^+(D)$ denotes \emph{the maximum out-degree} and $\Delta^-(D)$ denotes \emph{the maximum in-degree}, then $k \leq \Delta(D)+1$ for any $b$-coloring of $D$ using $k$ colors where $\Delta(D)=\min\{\Delta^+(D),\Delta^-(D)\}$ ( or simply by $\Delta$, $\Delta^+$ and $\Delta^-$  when the digraph $D$ is understood). In particular, 
\begin{equation}\label{eq2}
    dib(D) \leq \Delta+1.
\end{equation}

Moreover, for an induced subdigraph $H$ of a digraph $D$,
\begin{equation}\label{eq3}
    dib(H) \leq dib(D).
\end{equation}

A dart $uv\in A(D)$ is symmetric if $vu\in A(D)$ and asymmetric if $vu\notin A(D)$. Two vertices $u,v$ are adjacent if the dart $uv$ is symmetric. A digraph is symmetric (resp. asymmetric) if every arc of $D$ is symmetric (resp. asymmetric). If $D$ is a symmetric digraph, it can be considered as a graph, and then $dib(D)=b(D)$.

For example, \emph{complete symmetric digraph} of order $n$, denoted by $K_n$, has both darts $(u,v)$ and $(v,u)$ for all two distinct vertices $u$ and $v$.
Then $dib(K_n)=n$ because $b(K_n)=n$.

This paper is organized into three sections: The second section contains general bounds such as the Nordhaus-Gaddum relations. The third section deals with tournaments and fourth section with the dib-chromatic number in regular digraphs. 


\section{Bounds}\label{section2}

A \textit{clique} of a digraph $D$ is an induced subdigraph of $D$ which is a complete symmetric subdigraph. A \textit{maximum clique} of $D$ is a clique such that there is no clique with more vertices, and the \textit{clique number} $\omega(D)$ of $D$ is the number of vertices in a maximum clique in $D$. Then $\omega(D)\leq dc(D)$.

Let $D$ be a digraph of $n$ vertices that are listed in some specified order. In a \emph{greedy coloring} of $D$, the vertices are colored with positive integers according to a greedy algorithm that assigns to the vertex under consideration the smallest available color. Therefore, a greedy coloring is also a complete coloring, see \cite{MR3875016,MR4426060}.

It is not hard to see that a greedy acyclic coloring of a graph $D$ using $k$ colors is a coloring of $D$ having the property that for every two colors $i$ and $j$ with $i<j$, every vertex colored $j$ has a neighbor colored $i$. Note that for each vertex $v$ in the color class $V_j$ and each color class $V_i$, with $i<j$, $N^+(v)\cap V_i\not=\emptyset$ and $N^-(v)\cap V_i\not=\emptyset$ then $k \leq \Delta(D)+1$, see also \cite{MR4426060}.

Let $\varsigma$ be a $b$-coloring of a digraph $D$ using $k$ colors. On the one hand, let $u_i$ be any $b^+$-vertex of color $i$; we say that the set $\{u_1,\dots,u_k\}$ is a \textit{positive basis} of $\varsigma$. On the other hand, let $v_i$ be any $b^-$-vertex colored $i$; we say that the set $\{v_1,\dots,v_k\}$ is a \textit{negative basis} of $\varsigma$. A $b$-coloring may have many positive and negative bases, and the intersection between positive and negative bases can be empty or not.

The following result is a generalization of a result in \cite{MR1927071}.

\begin{theorem}\label{teo1}
Let $D$ be a digraph of order $n$. Assume that the vertices $u_1,\dots ,u_n$ of $D$ are ordered such that $deg^+(u_1)\geq deg^+(u_2) \geq \dots \geq deg^+(u_n)$. Let $t^+(D) := \max \{i \colon deg^+(u_i) \geq i-1\}$ be the maximum number $i$ such that $D$ contains at least $i$ vertices of out-degree at least $i-1$. Now, re-label the vertices as $v_1,\dots ,v_n$ such that $deg^-(v_1)\geq deg^-(v_2) \geq \dots \geq deg^-(v_n)$. Let $t^-(D) := \max \{i \colon deg^-(v_i) \geq i-1\}$ be the maximum number $i$ such that $D$ contains at least $i$ vertices of in-degree at least $i-1$. Therefore, \[dib(D) \leq t(D):=\min\{t^+(D),t^-(D)\}.\]
\end{theorem}
\begin{proof}
The proof is by contradiction. Suppose $dib(D)\geq t(D)+1$ then there exist at least $dib(D)$ vertices of out-degree and at least $dib(D)$ vertices of in-degree no smaller than \[dib(D)-1\geq t(D).\]
Therefore, \[deg^+(u_1),deg^+(u_2),\dots,deg^+(u_{t+1})\geq t(D)\] and \[deg^-(v_1),deg^-(v_2),\dots,deg^-(v_{t+1})\geq t(D),\] are a contradiction to the definition of $t(D)$.
\end{proof}

The \textit{complement} $D^c$ of a digraph $D$ is the digraph whose vertex set is $V(D)$ and where $uv$ is a dart of $D$ if and only if $uv$ is not a dart of $D^c$.

In \cite{MR3875016,MR4426060} were proven the Nordhaus-Gaddum relations for the dichromatic number, namely, $dc(D)+dc(D^c)\leq n+1$ for any digraph of order $n$. The result holds for the dib-chromatic number.

The following result is a generalization of a result in \cite{MR1927071}.

\begin{theorem}\label{teo2}
If $D$ is a digraph of order $n$, then
\[dib(D)+dib(D^c)\leq n+1.\]
\end{theorem}
\begin{proof}
Consider $b$-colorings of $D$ and $D^c$ with $dib(D)$ and $dib(D^c)$ colors, respectively. Observe that for any $b^+$-vertex $u$ in $D$, we have $deg^+_D(u)\geq dib(D)-1$ and then $deg^+_{D^c}(u)\leq n-dib(D)$. We denote the color class of $D$ with the color $i$ by $C_i$ for all $i\in\{1,2,\dots,dib(D)\}$ and $u_i$ is a $b^+$-vertex in $C_i$.

Case 1. There exists a color $i$ such that for every $u$ in $C_i$, $u$ is a $b^+$-vertex in $D^c$. In particular, $u_i$ is a $b^+$-vertex in $D$ and $D^c$, therefore $dib(D^c)-1\leq deg^+_{D^c}(u_i)\leq n-dib(D)$ and then the result follows.

Case 2. For every $i$, there exists $u$ in $C_i$ such that it is not a $b^+$-vertex in $D^c$. Then $dib(D^c)\leq\overset{dib(D)}{\underset{i=1}{\sum}}(|C_i|-1)\leq n-dib(D)$, and the result follows.
\end{proof}

An \textit{independent} set of vertices in a digraph $D$ is a set $I$ of vertices such that for every two vertices in $I$, there is no dart connecting the two. The \textit{independence number} $\beta(D)$ of $D$ is the largest possible order of an independent set.

The following result is a generalization of a result in \cite{MR1927071}.

\begin{corollary}\label{cor3}
For any digraph $D$ of order $n$, we have $dib(D)\leq n-\beta(D)+1.$
\end{corollary}
\begin{proof}
Since $dib(D^c)\geq \omega(D^c) = \beta(D)$, by Theorem \ref{teo2}, the result follows.
\end{proof}

Another upper bound is given in Theorem 1 of \cite{MR3875016}, which states that for any digraph $D$ of $m$ darts,  $dib(D)\leq \frac{1+\sqrt{1+4m}}{2}$ since any $b$-coloring of a digraph $D$ is complete.

We give two more upper bounds of $dib(D)$ for any digraph $D$ through the parameter $dac(D)$. First, we give a generalization of a result of \cite{MR1108075}.

\begin{theorem}\label{teo4}
For every digraph $D$ of order $n$, \[dac(D)\leq \frac{n+\omega(D)}{2}.\]
\end{theorem}
\begin{proof}
Suppose that $\omega(D)=k$ and $dac(G)=l$. Then $1\leq k \leq l$. So, there exists a complete acyclic coloring of $D$ with $l$ colors, but no complete acyclic coloring of $D$. Hence $V(D)$ can be partitioned into $l$ color classes $V_1,V_2,\dots,V_l$ such that for every two distinct integers $i,j\in\{1,2,\dots,l\}$, there is a vertex $u$ of $V_i$ and a $v$ vertex of $V_j$ such that $(u,v)$ is a dart of $D$. Since $\omega(D)=k$, at most $k$ of the sets $V_1,V_2,\dots,V_l$ can consist of a single vertex. Therefore, at least $l-k$ of these sets consist of two or more vertices, and so $n \geq 2(l-k) + k$. Therefore, $n \geq 2l-k$. Since $k\geq 1$, the result follows.
\end{proof}

The \textit{acyclic number} $\mathcal{A}(D)$ of $D$ is the largest possible order of an acyclic subdigraph of $D$. Then $\beta(D)\leq \mathcal{A}(D)$.

\begin{theorem}\label{teo5}
For every digraph $D$ of order $n$,
\[\frac{n}{\mathcal{A}(D)}\leq dc(D)\leq n-\mathcal{A}(D)+1.\]
\end{theorem}
\begin{proof}
Suppose that $dc(D)=k$ and let $\varsigma$ be an acyclic coloring of $D$ with $k$ colors, and $V_1,V_2,\dots ,V_k$ a partition given by $\varsigma$. Since
\[n=\overset{k}{\underset{i=1}{\sum}}|V_i|\leq k\mathcal{A}(D),\]
it follows the lower bound.

Next, let $A$ be an acyclic digraph of $\mathcal{A}(D)$ vertices and assign the color $1$ to each vertex of $A$. We color the remaining vertices of the digraph following greedy acyclic coloring. Hence
$dc(D)\leq |V(D)-A|+1$ and the upper bound follows.
\end{proof}

Finally, we give the following bound.

\begin{corollary}\label{cor6}
For every digraph $D$ of order $n$,
\[dac(D)\leq \left\lceil n-\frac{\mathcal{A}(D)}{2}\right\rceil.\]
\end{corollary}
\begin{proof}
Since $dac(D)\leq \frac{n+\omega(D)}{2}$ and $\omega(D)\leq dc(D)\leq n-\mathcal{A}(D)+1$ then
$dac(D)\leq \frac{2n-\mathcal{A}(D)+1}{2}$,
\[dac(D)\leq n-\frac{\mathcal{A}(D)}{2}+\frac{1}{2}\] and the result follows.
\end{proof}

Corollary \ref{cor7} shows that the bound of Corollary \ref{cor6} is tight.

The \emph{converse} of a digraph $D$ is the digraph $D^{op}$ which one obtains from $D$ by reversing all darts. Observe that $dib(D)=dib(D^{op})$ for any digraph $D$.


\section{On tournaments}

A \emph{tournament} $T_n$ is an orientation of the complete graph $K_n$. A tournament such that it is acyclic is called \emph{transitive}.

\begin{corollary}\label{cor7}
Let $T$ be a transitive tournament of order $n$. Then \[dib(T)=
\left\lceil {n}/{2}\right\rceil.\]
\end{corollary}
\begin{proof}

Let $(v_1,v_2,\dots,v_{n})$ be the acyclic order of its vertex set. The coloring $\varsigma(v_i)=\varsigma(v_{n+1-i})=i$ for $i\in \{1,\dots,\left\lceil \frac{n}{2}\right\rceil\}$ is a $b$-coloring of $T$ because the vertex $v_i$ is incident to the vertices $\{v_{\left\lceil\frac{n}{2}\right\rceil},\dots,v_n\}$ that contain all the colors, and $v_{n+1-i}$ is incident from the vertices $\{v_1,\dots,v_{\left\lfloor\frac{n}{2}\right\rfloor}\}$ which contains all the colors. Since $\mathcal{A}(T)=n$, by Corollary \ref{cor6} we have $dib(T)\leq
\left\lceil {n}/{2}\right\rceil$ and the result follows.
\end{proof}

\begin{corollary}\label{cor8}
If $T$ is a tournament of order $n$, then $n/2\leq dc(T)dib(T)$.
\end{corollary}
\begin{proof}
Let $\varsigma$ be a $dc(T)$-coloring and $x$ the order of the largest chromatic class $k$, then, $n\leq xdc(T)$. By Equation \ref{eq3} and Corollaries \ref{cor6} and \ref{cor7}, we have $\frac{x}{2} \leq dac(T)$ due to the subdigraph induced by $\varsigma^{-1}(k)$ is a transitive tournament.
\end{proof}

Let $\mathbb{Z}_{n}$ be the cyclic group of integers modulo $n$ $(n\geq 3)$ and $J$ a nonempty subset of $\mathbb{Z}_{n}\setminus \{0\}$ such that $\left\vert \{-j,j\}\cap J\right\vert =1$ for every $j\in J$ (then $|J|\leq \left\lfloor \frac{n-1}{2}\right\rfloor$). The \emph{circulant digraph} $\overrightarrow{C}_{n}(J)$ has vertex-set $V(\overrightarrow{C}_{n}(J))=\mathbb{Z}_{n}$ and dart-set 
\[E(\overrightarrow{C}_{n}(J))=\left\{ (i,j):i,j\in \mathbb{Z}_{n},j-i\in J\right\}.\]

\begin{corollary}\label{cor9}
Let $\overrightarrow{C}_{2m+1}(J)$ be a circulant tournament of order $2m+1$. If $J=\{1,2,\dots,m\}$, then \[dib(\overrightarrow{C}_{2m+1}(J))=m+1.\]
\end{corollary}
\begin{proof}
The coloring $\varsigma(0)=0$ and $\varsigma(i)=\varsigma(i+m)=i$ for $i\in \{1,2,\ldots,m\}$ defines a $b$-coloring of $\overrightarrow{C}_{2m+1}(J)$ because the vertex $0$ is incident to the vertices  $1,2,\ldots, m$ with colors $1,2,\ldots, m$ respectively, and the vertex $0$ is incident from the vertices  $m+1,m+2,\ldots, 2m$ with colors $1,2,\ldots, m$ respectively. Moreover, the vertex $m+i$ (colored $i$) is incident to the vertices $V_i:=\{m+1+i,m+2+i, \ldots, 2m+i=i-1\}$ with colors $i+1,i+2,\ldots, i-1$. Finally, the vertex $i$ (colored $i$) is incident from the set of vertices $V_i$.
Since $\Delta(\overrightarrow{C}_{2m+1}(J))=m$, by Equation \ref{eq2} the result follows.



\end{proof}

\begin{corollary}\label{cor10}
Let $\overrightarrow{C}_{n}(J)$ be a circulant digraph of order $n\geq 4$. If $J=\{1,2, \dots, k \}$, with $ 1\leq k \leq \left\lfloor \frac{n-2}{2}\right\rfloor$ then \[dib(\overrightarrow{C}_{n}(J))= k+1.\]
\end{corollary}
\begin{proof}
We define a coloring $\varsigma$ using $k+1$ colors as follows,  $\varsigma(t(k+1)+i)=i$ for $t\in \{0,1,\ldots, \left\lfloor \frac{n}{k+1}-1\right\rfloor\}$ and $i\in \{0,1,\ldots, k\}$. Let $r$ be the remainder of $\frac{n}{k+1}$. If $r>0$, then $\varsigma(t(k+1)+i)=i$ for $t=\left\lfloor \frac{n}{k+1}\right\rfloor$ and $i\in \{0,1,\ldots, r-1\}$.

Such coloring $\varsigma$ defines a $b$-coloring of $\overrightarrow{C}_{n}(J)$ because the vertex $i$ is incident to the vertices $i+1,i+2, \ldots, i+k$  for  $i\in \{0,1, \ldots, k\}$   and  in the vertex 
 $(k+i+1)$ is incident from the vertices  $i+1,i+2,\ldots, i+k$ for each $i\in \{0,1, \ldots, k\}$. Since $\Delta(\overrightarrow{C}_{n}(J))$ is $k$, by Equation \ref{eq2} the result follows.
 \end{proof}


A digraph is \textit{strongly connected} if for every pair of vertices $u$ and $v$ there is a directed $uv$-path. The strongly connected components of a digraph $D$ induced equivalence relation $\sim$ into subdigraphs that are themselves strongly connected, therefore we have the digraph $\tilde{D}:= D/\sim$.

\begin{theorem}\cite{MR2107429}
    If $T$ is a tournament with (exactly) $k$ strongly connected components, then $\tilde{T}$ is the transitive tournament of order $k$.
\end{theorem}

\begin{corollary}
    If $T$ is a tournament with (exactly) $k$ strongly connected components, then $k/2\leq dib(T)$.
\end{corollary}

Let $H$ be a digraph. The digraph $D$ is $H$-free if $D$ has no subdigraph isomorphic to $H$. A tournament $H$ is a \textit{hero} if and only if there exists $\alpha>0$ such that every $H$-free tournament $D$ has a transitive subset of cardinality at least $\alpha |V (D)|$, see \cite{MR2995716}.

\begin{corollary}
If $H$ is a hero and $D$ is an $H$-free tournament, then $dib(D)=\Theta(n)$.
\end{corollary}


\section{On regular digraphs}

In this section, we state some results about regular digraphs. We recall that the number of darts in a shortest weak $uv$-path (not necessarily a directed path) is denoted by $d(u,v)$, and if $u$ and $v$ belong to different components of a week disconnected digraph, then $d(u,v)= \infty$.

The following two theorems are generalizations of results in \cite{MR2063820}.

\begin{theorem}\label{teo9}
Let $D$ be a digraph, with $\Delta(D)=\Delta$, containing sets of vertices $B^+=\{u_1,\ldots, u_{\Delta+1}\}$ and $B^-=\{v_1, \ldots, v_{\Delta+1}\}$ with $deg^{+}(u_i)=\Delta$ and $deg^{-}(v_i)=\Delta$ for all $i\in\{1,\ldots,\Delta+1\}$ such that $d(x, y)\geq 4$ for any $x,y\in B^+\cup B^-$. Therefore, $dib(D)= \Delta+1$.
\end{theorem}
\begin{proof}
First, we proceed to give a coloring. If $|B^+\cap B^- | = l$, assume that they are the first $l$ elements of each set, i.e., $u_1=v_1,\ldots,u_l=v_l$ for $0\leq l\leq \Delta+1$.

The color $i$ is assigned to the vertices $u_i$ and $v_i$, for all $i\in \{1,\ldots,\Delta+1\}$. Then, for each dart $u_iz$, the vertex $z$ is colored by a color of $\{1,\ldots, \Delta+1 \}-\{i\}$ such that the vertices incident from $u_i$ have different colors. Next, for each dart $zv_i$ the vertex $z$ (consider only uncolored vertices) is colored with a color from $\{1,\ldots, \Delta+1 \}-\{i\}$ such that the vertices incident from $u_i$ have different colors.

Finally, we color the remaining vertices of the digraph following a greedy acyclic coloring. Therefore, each color class is acyclic and the result follows.
\end{proof}

A digraph $D$ is called \emph{$r$-regular} if $deg^+(v)=deg^-(v)=r$ for every vertex $v\in D$.

\begin{theorem}\label{teo10}
Let $D$ be an $r$-regular digraph  with $r\geq 2$ and at least $8r^4$ vertices then $dib(D)=r+1$.
\end{theorem}

\begin{proof} We find $r+1$ vertices $v_1, \ldots, v_{r+1}$ such that $d(v_i, v_j)\geq 4$  for all $i \neq j$ as follows.
Let $v_1$ be an arbitrary vertex and delete it together with the first, second, and third neighborhoods, then the number of deleted vertices is at most $1+2r+2r(2r-1)+2r(2r-1)^2 = 1+2r-4r^2+8r^3\in o(r^4)$. Now, we choose a vertex $v_2$ (then $d(v_2, v_1)\geq 4$ in $D$) and delete it together with the first, second, and third neighborhood, then the total number of vertices deleted is at most $ 1+2r-4r^2+8r^3$. So, we can repeat this process for vertices $v_3, \ldots, v_{r}$. Since $r(1+2r-4r^2+8r^3) < 8r^4$, there remains at least a vertex $v_{r+1}$. Due to $d(v_i, v_j)\geq 4$ for $i>j$ and by Theorem \ref{teo9}, $dib(D)= r+1$.

\end{proof}



The following remarks are inspired by the paper \cite{MR2606622}. Theorem \ref{teo10} implies that for a given $r$, there is only a finite number of $r$-regular digraphs $D$ with $dib(D)<r+1$. 

A $1$-regular digraph is a set of directed cycles. It is not hard to see that a directed cycle has a dib-chromatic number 2. Hence, there are no $1$-regular digraphs $D$ with $dib(D)<2$.

For the case of $r=2$, the upper bound indicates that any $2$-regular digraph of $n\geq 106$ vertices has a dib-chromatic number 3. However, this bound could be far from the optimal value. 

We use the database of small-order graphs (with vertices of small degrees) that appears in \cite{MR2973372} to generate the following $2$-regular digraphs of dib-chromatic number $2$. Figure \ref{Fig1} displays such digraphs.

\begin{figure}[htbp!]
\begin{center}
\includegraphics{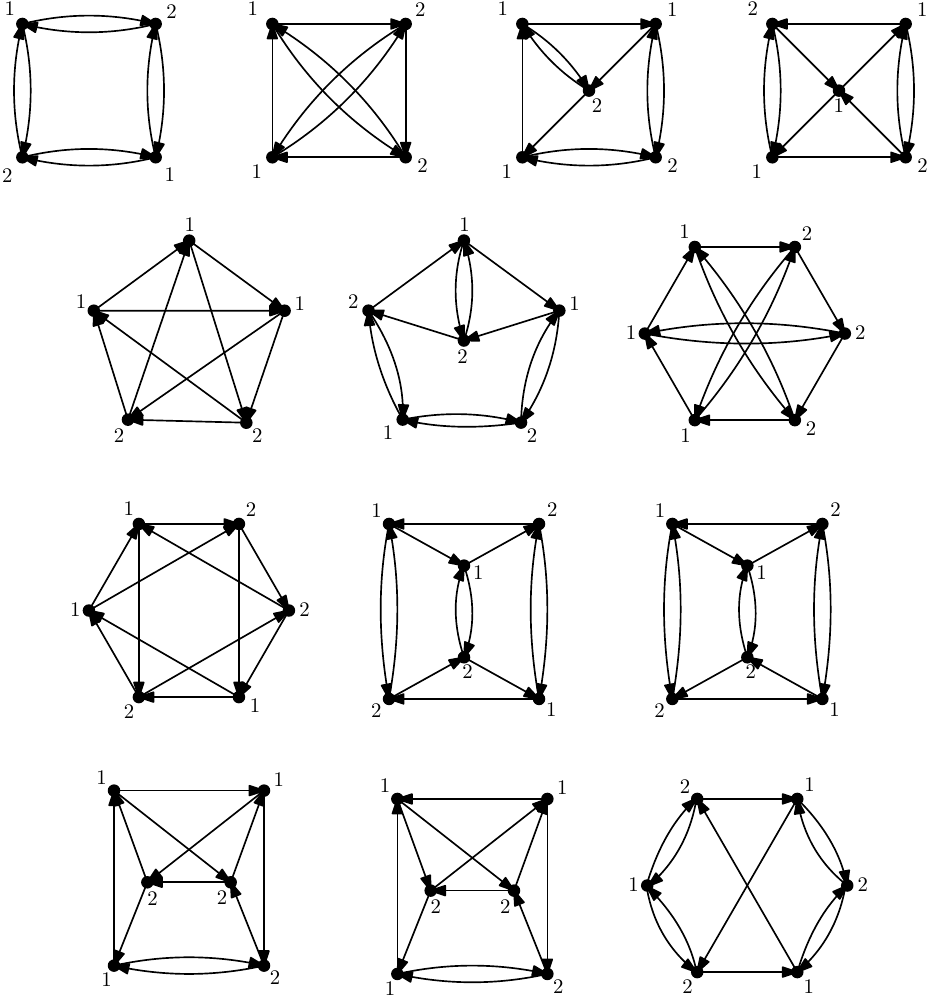}
\caption{\label{Fig1} For each digraph $D$, $dib(D)=2$.}
\end{center}
\end{figure}

 We conjecture that the remaining $2$-regular digraphs have dib-chromatic number 3. 
 
 Another nice problem could be to generate the set of digraphs of small order with vertices of small degrees (with digons).


\section{Statements and Declarations}

This work was partially supported by PAIDI-M{\' e}xico under Project PAIDI/007/21.

The authors have no relevant financial or non-financial interests to disclose.

All authors contributed to the conception and design of the study. The first draft of the manuscript was written by Christian Rubio-Montiel and all authors commented on previous versions of the manuscript. All authors read and approved the final manuscript.


\bibliographystyle{plain}
\bibliography{biblio}

\end{document}